\documentclass[11pt]{article}

\usepackage[T1]{fontenc}
\usepackage{lmodern}
\usepackage{amsmath,amssymb,amsthm,mathtools}
\usepackage{microtype}
\usepackage{geometry}
\usepackage{hyperref}

\hypersetup{
  pdftitle={Strongly Consistent Estimation of the Extended l1-Sum of
    phi-Mixing Coefficients from a Single Trajectory},
  pdfauthor={Senhan Yao},
  colorlinks=true,
  linkcolor=blue,
  citecolor=blue,
  urlcolor=blue
}

\newtheorem{theorem}{Theorem}[section]
\newtheorem{proposition}[theorem]{Proposition}
\newtheorem{lemma}[theorem]{Lemma}
\newtheorem{corollary}[theorem]{Corollary}

\theoremstyle{remark}
\newtheorem{remark}[theorem]{Remark}

\newcommand{\Pp}{\mathbb{P}}
\newcommand{\E}{\mathbb{E}}
\newcommand{\N}{\mathbb{N}}
\newcommand{\R}{\mathbb{R}}
\newcommand{\1}{\mathbf{1}}
\newcommand{\cF}{\mathcal{F}}
\newcommand{\cD}{\mathcal{D}}
\newcommand{\cG}{\mathcal{G}}
\newcommand{\cP}{\mathcal{P}}
\newcommand{\cA}{\mathcal{A}}
\newcommand{\cB}{\mathcal{B}}
\newcommand{\cC}{\mathcal{C}}
\newcommand{\cH}{\mathcal{H}}
\newcommand{\cM}{\mathcal{M}}
\newcommand{\abs}[1]{\left|#1\right|}

\title{Strongly Consistent Estimation of the Extended $\ell_1$-Sum of
$\phi$-Mixing Coefficients from a Single Trajectory}
\author{Senhan Yao}
\date{August 12, 2026}

\begin{document}
\maketitle

\begin{abstract}
Khaleghi and Lugosi asked whether the sum of the $\phi$-mixing
(uniform-mixing) coefficients of a real-valued discrete-time stationary
ergodic process can be consistently estimated from a single sample path.
We construct a deterministic sequence of Borel statistics that, for every
such process, converges almost surely to the extended sum
$\sum_{m\geq1}\phi(m)$, including divergence to $+\infty$ when the sum is
infinite.  The estimator combines finite dyadic cylinders with a vanishing
empirical cutoff on conditioning-event frequencies.  When the target is
finite, $\alpha(m)\leq\phi(m)$ supplies summable covariance control for the
growing finite classes, and a stable-division estimate yields the upper
bound.  Fixed positive-probability witnesses and Birkhoff's theorem give the
lower bound, including the infinite-target case.  No mixing rate or known
positive lower bound on conditioning-event probabilities is required.
\end{abstract}

\section{Introduction}

Let $X=(X_t)_{t\geq1}$ be a stationary ergodic process.  Mixing coefficients
quantify residual dependence between events separated in time; standard
references include Doukhan~\cite{Doukhan1994} and
Bradley~\cite{Bradley2005,Bradley2007}.  Nobel~\cite{Nobel2006} used
hypothesis testing to study polynomial decay rates for covariance-based
mixing conditions.  Khaleghi and Lugosi~\cite{KhaleghiLugosi2023}
constructed universal strongly consistent estimators of the $\ell_1$-norms
of the $\alpha$- and $\beta$-mixing coefficient sequences from a single
real-valued stationary ergodic sample path, but explicitly left the
analogous $\phi$-problem open.  Theorem~\ref{thm:main} gives a positive
answer to that question in the same unrestricted observation model; see
\cite[Sec.~I and Sec.~V (Outlook)]{KhaleghiLugosi2023}.  We are not aware
of a previous result giving a law-independent strongly consistent estimator
of this $\phi$-mixing $\ell_1$-sum under only stationarity and ergodicity.  They identify
conditioning on potentially rare events, whose probabilities may be
arbitrarily small, as the central obstruction.  A related plug-in perspective
for conditional probabilities is developed by
Gr\"unew\"alder~\cite{Grunewalder2018}.  For context,
Adams and Nobel~\cite{AdamsNobel2010} proved uniform convergence of relative
frequencies over countable finite-VC classes under stationary ergodic sampling
without mixing assumptions; their result concerns a fixed class, whereas the
difficulty here is uniform control along deterministic growing cylinder
classes together with conditional-probability ratios.

Specifically, we construct deterministic Borel statistics
$\widehat\Phi_n(X_1,\ldots,X_n)$ such that, for every real-valued stationary
ergodic process,
\[
  \widehat\Phi_n
  \longrightarrow
  \Phi:=\sum_{m=1}^{\infty}\phi(m)
  \qquad\text{almost surely in }[0,\infty].
\]
The same deterministic maps are used for every law; the probability-one
convergence set may depend on the law.  The construction also detects the
nonsummable case by diverging to $+\infty$.

\paragraph{Central mechanism.}
If $\Phi<\infty$, then $\alpha(m)\leq\phi(m)$ supplies summable covariance
control for the growing finite classes.  A vanishing empirical cutoff
$q_c:=2^{-c}$ and the inequality $\Pp(A\cap B)\leq\Pp(A)$ stabilize the
conditional-probability ratios without requiring a known population lower
bound.  The lower bound instead uses fixed positive-probability
finite-cylinder witnesses and Birkhoff's theorem; no distribution-free
ergodic rate is needed~\cite{Shields1996}, and finite partial sums also force
divergence when $\Phi=\infty$.

\paragraph{Relation to prior constructions.}
The dyadic-cylinder and plug-in architecture builds on Khaleghi and
Lugosi~\cite[Secs.~III-A--III-B and Sec.~V (Outlook)]{KhaleghiLugosi2023},
who identify rare conditioning events as the central obstacle.  A key
ingredient in the present construction is the combination of the vanishing
\emph{empirical} rare-event cutoff with the stable-division estimate, together
with summable-$\alpha$ control only in the finite-target branch and
fixed-witness ergodic recovery for the lower bound.  The remaining tools---Birkhoff's
theorem~\cite{Birkhoff1931}, Chebyshev's inequality, the union bound, and
Borel--Cantelli---are standard.

\section{Definitions and main theorem}

Throughout, $\N=\{1,2,\ldots\}$, $(\Omega,\cF,\Pp)$ is a
probability space, and $X=(X_t)_{t\geq1}$ is a real-valued stationary
ergodic process.  For integers $1\leq a\leq b$, write
\[
  X_a^b:=(X_a,\ldots,X_b),
  \qquad
  \cF_a^b:=\sigma(X_a,\ldots,X_b),
  \qquad
  \cF_a^\infty:=\sigma(X_a,X_{a+1},\ldots).
\]
Stationarity means invariance of all finite-dimensional distributions under
a common time shift.  Ergodicity means that the left shift on the trajectory
space, equipped with the law of $(X_1,X_2,\ldots)$, has only invariant events
of probability zero or one.

For events $A,B\in\cF$ with $\Pp(A)>0$, write
$\Pp(B\mid A):=\Pp(A\cap B)/\Pp(A)$.  For two sub-$\sigma$-fields
$\cA,\cB\subseteq\cF$, define
\begin{align}
  \alpha(\cA,\cB)
  &:=
  \sup_{A\in\cA,\;B\in\cB}
  \abs{\Pp(A\cap B)-\Pp(A)\Pp(B)},                                     \label{eq:alpha-sigma}\\
  \phi(\cA,\cB)
  &:=
  \sup_{\substack{A\in\cA,\;\Pp(A)>0\\B\in\cB}}
  \abs{\Pp(B\mid A)-\Pp(B)}.                                          \label{eq:phi-sigma}
\end{align}
The second coefficient is directional: the event in $\cA$ is the
conditioning event.  The strong-mixing coefficient $\alpha$ goes back to
Rosenblatt~\cite{Rosenblatt1956}, while the uniform-mixing coefficient
$\phi$ is classically associated with
Ibragimov~\cite{Ibragimov1962}; see Doukhan~\cite{Doukhan1994} and
Bradley~\cite{Bradley2005,Bradley2007} for systematic accounts.

For the main theorem we use the indexing appearing in the displayed
process definitions of Khaleghi and Lugosi~\cite{KhaleghiLugosi2023}.  Thus,
for $m\geq1$, define
\begin{align}
  \alpha_X(m)
  &:=
  \sup_{j\geq1}\alpha(\cF_1^j,\cF_{j+m}^\infty),                       \label{eq:alpha-m}\\
  \phi_X(m)
  &:=
  \sup_{j\geq1}\phi(\cF_1^j,\cF_{j+m}^\infty),                         \label{eq:phi-m}
\end{align}
Under this convention the first future coordinate is
$j+m$; in particular, $m=1$ corresponds to adjacent past and future
blocks.  We suppress the subscript $X$ when no ambiguity is possible and
write
\[
  \Phi_X:=\sum_{m=1}^{\infty}\phi_X(m)\in[0,\infty],
\]
which we call the extended $\ell_1$-sum of the $\phi$-mixing coefficients.
In particular, $\Phi_X<\infty$ implies $\phi_X(m)\to0$ as $m\to\infty$,
so the finite-target case lies in the usual $\phi$-mixing class.
Since every conditional-dependence score in \eqref{eq:phi-sigma} belongs
to $[0,1]$, each $\phi_X(m)\in[0,1]$ and the extended sum above is
well-defined.  When $\Phi_X<\infty$, this is the ordinary $\ell_1$-norm of
the coefficient sequence; when $\Phi_X=\infty$, the sequence is not an
element of $\ell_1$, which is why we use the term \emph{extended
$\ell_1$-sum} for the target throughout.

The observation model and target above match the substantive
$\phi$-question posed by Khaleghi and Lugosi
\cite[Sec.~I]{KhaleghiLugosi2023}: one finite prefix of a
single real-valued discrete-time stationary ergodic trajectory, with no
structural assumption on the law.  There is an apparent one-index
discrepancy in their presentation: the displayed process definitions place
the future at $j+m$, whereas some later finite-block formulas use $j+m+1$.
To avoid depending on which indexing was intended,
Corollary~\ref{cor:shifted} proves universal strong consistency for every
fixed finite lag shift and therefore covers both placements.  We additionally
treat the target as an extended nonnegative sum and require divergence of the
estimator when the series is infinite.

\begin{theorem}[Universal strong consistency]\label{thm:main}
There exists a deterministic sequence of Borel measurable functions
\[
  \widehat\Phi_n:\R^n\to[0,\infty)
\]
such that, for every real-valued discrete-time stationary ergodic process $X$,
\[
  \widehat\Phi_n(X_1,\ldots,X_n)
  \longrightarrow
  \Phi_X
  \qquad\text{$\Pp$-almost surely},
\]
where the convergence is in the extended half-line $[0,\infty]$.
In particular, if $\Phi_X=\infty$, then
\[
  \widehat\Phi_n(X_1,\ldots,X_n)\longrightarrow+\infty
  \qquad\text{almost surely},
\]
meaning that for every finite $R$ the inequality
$\widehat\Phi_n(X_1,\ldots,X_n)>R$ holds eventually almost surely.
\end{theorem}

Thus, in the finite-sum case, Theorem~\ref{thm:main} provides a positive answer to the $\phi$-mixing $\ell_1$-estimation question posed by Khaleghi and Lugosi~\cite{KhaleghiLugosi2023} under the same stationary-ergodic single-trajectory observation model.

\begin{remark}[Robustness to fixed lag shifts]\label{rem:lag-shift}
For a fixed $d\in\{0,1,2,\ldots\}$, one may instead place the future
at $j+m+d$.  Corollary~\ref{cor:shifted} proves universal strong consistency
for the corresponding extended sum.  In particular, the construction covers
both the $j+m$ placement in the displayed process definitions of
Khaleghi and Lugosi and the $j+m+1$ placement appearing in some of their
finite-block formulas.
\end{remark}

\begin{remark}[One-sided and two-sided conventions]\label{rem:two-sided}
The one-sided coefficient in \eqref{eq:phi-m} coincides with the usual
stationary two-sided coefficient of any stationary two-sided extension:
\[
  \phi_X(m)
  =\phi\bigl(\sigma(\widetilde X_t:t\leq0),
                \sigma(\widetilde X_t:t\geq m)\bigr).
\]
In particular, the target is intrinsic to the one-sided stationary law and
does not depend on the chosen extension.  Appendix~\ref{app:two-sided}
gives the construction and proof.
\end{remark}

\paragraph{Related work and exact scope.}
Nearby estimation results differ in observation model, target, or structural
assumptions.  Ahsen and Vidyasagar~\cite{AhsenVidyasagar2014} estimate
$\alpha$-, $\beta$-, and $\phi$-dependence coefficients between a pair of
random variables from independent paired observations.  Gr\"unewalder and
Khaleghi~\cite{GrunewalderKhaleghi2026} estimate individual $\beta$-mixing
coefficients from one trajectory of a stationary geometrically ergodic
Markov process, with additional smoothness assumptions in the real-valued
case.

Arvanitis~\cite[Sec.~2]{Arvanitis2023}, in a kernel-density setting for
stationary uniformly ($\phi$-) mixing processes, assumes an absolutely
summable sequence of $\phi$-mixing coefficients and derives non-asymptotic
concentration bounds.  For a confidence-set application, a known upper bound
on the sum of the mixing coefficients is imposed; removing that restriction
by estimating the mixing coefficients is left for future research, with the
observation that estimation of the mixing-coefficient series might be
facilitated by Ahsen and Vidyasagar together with truncation.  This does not
solve the problem considered here: no universal single-trajectory estimator
of the extended sum $\sum_{m\geq1}\phi(m)$ is constructed or proved
consistent there, and both the inferential target and the accompanying
density/kernel regularity assumptions are different.

Qi, Shen, and Zheng~\cite[Appendix~D, eqs.~(65)--(66), Theorem~5, and
Remarks~1--2]{QiShenZheng2024} propose, for a fixed integer lag $k$, a
histogram estimator $\widehat\phi_d(k)$ of the finite-block quantity
$\phi_d(k)$.  Their Theorem~5---described there as adapted from a result of
McDonald et al.~\cite{McDonaldEtAl2015}---is stated as a consistency result
for $\widehat\phi_{d_n}(k)$ relative to the full coefficient
$\phi(k)$ under bounded support, sufficient smoothness of the density, and a
density bounded away from zero, with tuning sequences satisfying
\[
  n h_n^{d_n}\to\infty,\qquad d_n h_n\to0,\qquad
  d_n\to\infty,\qquad h_n\to0.
\]
The published supplementary PDF literally states the threshold
$|\widehat\phi_{d_n}(k)-\phi(k)|>0$ immediately after the phrase ``for any
$\epsilon>0$''.  The same display appears in the authors' SSRN manuscript.
Because $\epsilon$ is otherwise unused in that statement and the cited
McDonald et al.~consistency result is convergence in probability, the
threshold $>0$ appears to be a typographical error, presumably intended as
$>\epsilon$.  We do not rely on that display or on this interpretation; we
use the result here only to record the estimator, its stated regularity
assumptions, and its intended fixed-lag consistency claim.  Remark~1 explains
that the $\phi$-case requires handling conditional densities in addition to
the joint-density arguments used for $\beta$-mixing, while Remark~2 calls
bounded support a relatively strong condition and notes that their techniques
do not cover the unbounded-support extension.  Thus both the assumptions and
the fixed-lag target differ from Theorem~\ref{thm:main}.

Later papers continued to flag the broader problem of estimating
$\phi$-mixing coefficients themselves from a fully observed stationary
trajectory.  Khaleghi~\cite[Sec.~IV (Outlook)]{Khaleghi2025} notes that,
even with full observations, rare conditioning events obstruct coefficient
estimation and that it remains unclear whether the $\phi$-mixing coefficients
can be consistently estimated from stationary sample paths.  Karagulyan and
Alquier~\cite[Appendix~C]{KaragulyanAlquier2026} likewise state that an
empirical non-Markov PAC--Bayes bound based on $\phi$ would require estimating
the $\phi$-mixing coefficients, which they describe as an open question.
These later statements concern the broader coefficient-estimation problem,
not specifically the scalar $\ell_1$-sum.  The closest comparison for
Theorem~\ref{thm:main} is therefore Khaleghi and
Lugosi~\cite[Sec.~I and Sec.~V]{KhaleghiLugosi2023}, whose target and
single-trajectory stationary-ergodic observation model match the present
setting.

The theorem concerns the scalar extended sum.  We do \emph{not} claim that
$\widehat\phi_c(m)$ converges to every individual $\phi(m)$ under bare
ergodicity: Proposition~\ref{prop:pointwise-lower} gives the universal
$\liminf$ lower bound, while Corollary~\ref{cor:coordinate-consistency} gives
full coordinatewise convergence when $\Phi<\infty$.

\section{Reduction to a bounded state space}

The use of dyadic partitions is most convenient on $(0,1)$ and causes no
loss of generality.

\begin{proposition}[Borel-isomorphic reduction]\label{prop:reduction}
It suffices to prove Theorem~\ref{thm:main} for stationary ergodic processes with
values in $(0,1)$.
\end{proposition}

\begin{proof}
Let
\[
  h(x):=\frac12+\frac1\pi\arctan x,\qquad x\in\R.
\]
Then $h:\R\to(0,1)$ is a Borel bijection with Borel inverse.  Define
$Y_t=h(X_t)$.  For every index set $I\subseteq\N$,
\[
  \sigma(Y_t:t\in I)=\sigma(X_t:t\in I).
\]
Hence every $\alpha$- and $\phi$-coefficient is unchanged:
\[
  \alpha_Y(m)=\alpha_X(m),
  \qquad
  \phi_Y(m)=\phi_X(m).
\]
Let $H:\R^{\N}\to(0,1)^{\N}$ be the coordinatewise map
$H(x_1,x_2,\ldots)=(h(x_1),h(x_2),\ldots)$, and let $T$ denote the left
shift on either path space.  Then $H$ is a bimeasurable bijection and
$H\circ T=T\circ H$.  Hence shifted finite-dimensional distributions are
preserved and invariant events correspond under $H$; stationarity and
ergodicity therefore pass from $X$ to $Y$ (and conversely).  Thus an
estimator constructed for $Y_1,\ldots,Y_n$ yields one for
$X_1,\ldots,X_n$ by precomposing with $h$ in each coordinate.
\end{proof}

From now through Section~\ref{sec:completion} we work with a stationary ergodic
process $Y=(Y_t)_{t\geq1}$ taking values in $(0,1)$ and write simply
$\phi(m)$ and $\Phi$.  We also reuse the notation
\[
  \cF_a^b:=\sigma(Y_a,\ldots,Y_b),
  \qquad
  \cF_a^\infty:=\sigma(Y_a,Y_{a+1},\ldots),
\]
which is unambiguous because the coordinatewise Borel bijection in
Proposition~\ref{prop:reduction} preserves these $\sigma$-fields.

\section{Finite dyadic approximation of the \texorpdfstring{$\phi$}{phi}-coefficient}
\label{sec:approximation}

We first show carefully that finite dyadic cylinders recover the full
$\phi$-coefficient, despite the denominator in \eqref{eq:phi-sigma}.

\subsection{A general approximation lemma}

\begin{lemma}[Approximation by a generating algebra]\label{lem:measure-approx}
Let $\cC$ be an algebra of subsets of $\Omega$ and let
$\cH=\sigma(\cC)$.  For every $H\in\cH$ and every $\eta>0$, there exists
$C\in\cC$ such that
\[
  \Pp(H\triangle C)<\eta.
\]
\end{lemma}

\begin{proof}
Let
\[
  \cM:=
  \left\{
    H\in\cH:
    \inf_{C\in\cC}\Pp(H\triangle C)=0
  \right\}.
\]
Clearly $\cC\subseteq\cM$.  If $H\in\cM$, then, for any $C\in\cC$,
\[
  \Pp(H^c\triangle C^c)=\Pp(H\triangle C),
\]
so $H^c\in\cM$.

Now let $(H_i)_{i\geq1}\subseteq\cM$ and put $H=\bigcup_{i\geq1}H_i$.
Given $\eta>0$, continuity from below gives an $M$ such that
\[
  \Pp\!\left(H\setminus\bigcup_{i=1}^M H_i\right)<\frac{\eta}{2}.
\]
For each $1\leq i\leq M$, choose $C_i\in\cC$ with
\[
  \Pp(H_i\triangle C_i)<\frac{\eta}{2M}.
\]
Since $\cC$ is an algebra, $C:=\bigcup_{i=1}^M C_i\in\cC$, and
\[
  H\triangle C
  \subseteq
  \left(H\setminus\bigcup_{i=1}^M H_i\right)
  \cup
  \bigcup_{i=1}^M(H_i\triangle C_i).
\]
Thus $\Pp(H\triangle C)<\eta$.  Hence $\cM$ is a $\sigma$-field
containing $\cC$, and therefore $\cM=\sigma(\cC)=\cH$.
\end{proof}

\begin{lemma}[Continuity of the conditional-dependence functional]
\label{lem:conditional-continuity}
Suppose $A_n,A,B_n,B\in\cF$ satisfy
\[
  \Pp(A_n\triangle A)\to0,\qquad
  \Pp(B_n\triangle B)\to0,
  \qquad
  \Pp(A)>0.
\]
Then $\Pp(A_n)>0$ eventually and
\[
  \abs{\Pp(B_n\mid A_n)-\Pp(B_n)}
  \longrightarrow
  \abs{\Pp(B\mid A)-\Pp(B)}.
\]
\end{lemma}

\begin{proof}
We have
\[
  \abs{\Pp(A_n)-\Pp(A)}
  \leq \Pp(A_n\triangle A)\to0,
\]
so $\Pp(A_n)\to\Pp(A)>0$.  Also
\[
  (A_n\cap B_n)\triangle(A\cap B)
  \subseteq
  (A_n\triangle A)\cup(B_n\triangle B),
\]
and hence
\[
  \Pp(A_n\cap B_n)\to\Pp(A\cap B).
\]
Similarly $\Pp(B_n)\to\Pp(B)$.  Division by the eventually positive
quantity $\Pp(A_n)$ now gives
\[
  \frac{\Pp(A_n\cap B_n)}{\Pp(A_n)}
  \longrightarrow
  \frac{\Pp(A\cap B)}{\Pp(A)},
\]
and the claim follows by continuity of the absolute value.
\end{proof}

\begin{proposition}[Restriction of $\phi$ to generating algebras]
\label{prop:phi-algebra}
Let $\cC$ and $\cD$ be algebras with
$\sigma(\cC)=\cA$ and $\sigma(\cD)=\cB$.  Then
\[
  \phi(\cA,\cB)
  =
  \sup_{\substack{C\in\cC,\;\Pp(C)>0\\D\in\cD}}
  \abs{\Pp(D\mid C)-\Pp(D)}.
\]
\end{proposition}

\begin{proof}
The right-hand side is at most $\phi(\cA,\cB)$ because
$\cC\subseteq\cA$ and $\cD\subseteq\cB$.

For the reverse inequality, fix $A\in\cA$ with $\Pp(A)>0$ and
$B\in\cB$.  By Lemma~\ref{lem:measure-approx}, choose sequences
$C_n\in\cC$ and $D_n\in\cD$ such that
\[
  \Pp(C_n\triangle A)\to0,
  \qquad
  \Pp(D_n\triangle B)\to0.
\]
By Lemma~\ref{lem:conditional-continuity},
\[
  \abs{\Pp(D_n\mid C_n)-\Pp(D_n)}
  \to
  \abs{\Pp(B\mid A)-\Pp(B)}.
\]
Because $\Pp(C_n)>0$ eventually, every sufficiently late term is admissible
in the right-hand supremum.  Hence that supremum is at least
$\abs{\Pp(B\mid A)-\Pp(B)}$.  Taking the supremum over $A$ and $B$
proves the reverse inequality.
\end{proof}

\begin{remark}\label{rem:rare-population}
The denominator issue causes no difficulty in
Proposition~\ref{prop:phi-algebra}: for each fixed target accuracy one fixes one
near-optimal conditioning event $A$ with $\Pp(A)>0$.  Its probability may
be extremely small, but it is a fixed positive number.  No uniform
continuity as $\Pp(A)\downarrow0$ is asserted or needed.
\end{remark}

\subsection{Dyadic cylinder classes}

For $\ell\geq1$, let
\[
  \cP_\ell
  :=
  \left\{
    [a2^{-\ell},(a+1)2^{-\ell})\cap(0,1):
    a=0,\ldots,2^\ell-1
  \right\}.
\]
For a block length $r\geq1$, let $\cD_{r,\ell}$ be the finite
$\sigma$-field on $(0,1)^r$ generated by the product partition
$\cP_\ell^{\otimes r}$.  Thus
\[
  \abs{\cD_{r,\ell}}=2^{\,2^{r\ell}}.
\]
The classes are nested in $\ell$:
\[
  \cD_{r,\ell}\subseteq\cD_{r,\ell+1}.
\]

Fix $m,j,k\geq1$, $A\in\cD_{j,\ell}$ and
$B\in\cD_{k,\ell}$.  Put
\[
  r=r(m,j,k):=j+m+k-1.
\]
Within $(0,1)^r$, define the lifted events
\begin{align*}
  A^{\uparrow}
  &:=
  \left\{y_1^r:y_1^j\in A\right\},\\
  B^{\uparrow}
  &:=
  \left\{y_1^r:y_{j+m}^{j+m+k-1}\in B\right\},\\
  J_{m,j,k}(A,B)
  &:=
  A^\uparrow\cap B^\uparrow.
\end{align*}
All three belong to $\cD_{r,\ell}$.

\begin{lemma}[Common-span reduction]\label{lem:common-span}
Fix $c\geq1$ and $1\leq m,j,k\leq c$.  If
$A\in\cD_{j,c}$ and $B\in\cD_{k,c}$, then, with
\[
  r=j+m+k-1\leq3c,
\]
the three events
\[
  A^\uparrow,\qquad B^\uparrow,\qquad J_{m,j,k}(A,B)
\]
all belong to the single finite $\sigma$-field $\cD_{r,c}$.  Consequently
all three empirical frequencies entering the corresponding score are
coordinates indexed by $\cG_c$.
\end{lemma}

\begin{proof}
The lifted events constrain only coordinates of the common block
$y_1^r$, and every coordinate constraint is a union of atoms of
$\cP_c$.  Hence each lifted event belongs to $\cD_{r,c}$; the bound
$r\leq3c$ places each pair $(r,C)$ in $\cG_c$.
\end{proof}

Whenever a set
$C\subseteq(0,1)^r$ appears as the argument of $\Pp$ below, we use the
shorthand
\[
  \Pp(C):=\Pp(Y_1^r\in C),
\]
with the block length $r$ understood from context.  With this convention,
stationarity gives
\[
  \Pp(B^\uparrow)=\Pp(Y_1^k\in B).
\]
Whenever $\Pp(A^\uparrow)>0$, define
\begin{equation}\label{eq:population-score}
  d_m(A,B)
  :=
  \abs{
    \frac{\Pp(J_{m,j,k}(A,B))}{\Pp(A^\uparrow)}
    -
    \Pp(B^\uparrow)
  }.
\end{equation}

For $c\geq m$ and $0<q\leq1$, set
\begin{equation}\label{eq:population-truncated}
  \phi_c^{(q)}(m)
  :=
  \max_{\substack{
      1\leq j,k\leq c\\
      A\in\cD_{j,c},\,B\in\cD_{k,c}\\
      \Pp(A^\uparrow)\geq q
  }}
  d_m(A,B).
\end{equation}
The maximum exists because the search class is finite.  The set is
nonempty since the full-space event may be used for $A$.  The quantity
$\phi_c^{(q)}(m)$ is a population proof device only; it is not available to
or used by the empirical estimator.

\begin{proposition}[Population truncation loses nothing asymptotically]
\label{prop:population-truncation}
Let $(q_c)_{c\geq1}$ be any sequence in $(0,1]$ with $q_c\to0$.  Then, for
every fixed $m\geq1$,
\[
  \phi_c^{(q_c)}(m)\longrightarrow\phi(m).
\]
\end{proposition}

\begin{proof}
Every pair in \eqref{eq:population-truncated} is admissible in the
definition of $\phi(m)$, so
\[
  \phi_c^{(q_c)}(m)\leq\phi(m).
\]

For the reverse inequality, fix $\eta>0$.  For each $j\geq1$, define the
past dyadic algebra
\[
  \cC_j^-
  :=
  \bigcup_{\ell\geq1}
  \left\{
    \{Y_1^j\in A\}:A\in\cD_{j,\ell}
  \right\},
\]
and, with $s:=j+m$, define the future dyadic-cylinder class
\[
  \cC_s^+
  :=
  \bigcup_{k,\ell\geq1}
  \left\{
    \{Y_s^{s+k-1}\in B\}:B\in\cD_{k,\ell}
  \right\}.
\]
Because the classes $\cD_{j,\ell}$ are nested in $\ell$, $\cC_j^-$ is an
algebra.  The class $\cC_s^+$ is also an algebra: given two of its events,
choose a common block length $k$ and a common dyadic level $\ell$ by
adjoining unconstrained trailing coordinates and refining the two dyadic
partitions.  Both events are then represented by members of the same
finite $\sigma$-field $\cD_{k,\ell}$, so their union and complements are
again in $\cC_s^+$.

Moreover,
\[
  \sigma(\cC_j^-)=\cF_1^j,
  \qquad
  \sigma(\cC_s^+)=\cF_s^\infty.
\]
Indeed, dyadic intervals generate the Borel $\sigma$-field on $(0,1)$,
so the first identity follows from finite-dimensional dyadic rectangles.
For the second, if $t=s+d$ and $I$ is a dyadic interval, then
$\{Y_t\in I\}\in\cC_s^+$ by taking block length $d+1$ and leaving the
first $d$ coordinates unrestricted.  Thus $\sigma(\cC_s^+)$ contains
every coordinate $\sigma$-field $\sigma(Y_t)$ for $t\geq s$, while the
reverse inclusion is immediate from the definition.

By the definition of the supremum in \eqref{eq:phi-m}, choose $j$ such
that
\[
  \phi(\cF_1^j,\cF_{j+m}^\infty)
  >
  \phi(m)-\frac{\eta}{2}.
\]
By Proposition~\ref{prop:phi-algebra}, the latter supremum is the
supremum over finite dyadic-cylinder events.  Hence there exist a finite
future block length $k$, finite dyadic levels, and admissible events whose
conditional-dependence score exceeds
$\phi(\cF_1^j,\cF_{j+m}^\infty)-\eta/2$; no attainment of the supremum
is being assumed.  Refining the two dyadic levels
to a common $\ell$, we obtain
$A\in\cD_{j,\ell}$ and $B\in\cD_{k,\ell}$ such that
\[
  p:=\Pp(A^\uparrow)>0
\]
and
\[
  d_m(A,B)>\phi(m)-\eta.
\]
For all sufficiently large $c$ we have
$c\geq\max\{m,j,k,\ell\}$ and $q_c<p$.  Because the dyadic
$\sigma$-fields are nested, the same sets $A$ and $B$, not merely
approximations to them, belong to $\cD_{j,c}$ and $\cD_{k,c}$ for every
$c\geq\ell$.  Thus this same pair is included in
\eqref{eq:population-truncated}.  Hence
\[
  \liminf_{c\to\infty}\phi_c^{(q_c)}(m)
  \geq
  \phi(m)-\eta.
\]
Letting $\eta\downarrow0$ completes the proof.
\end{proof}

\section{Empirical cylinder probabilities and a variance bound}
\label{sec:concentration}

For $r\geq1$, a Borel set $C\subseteq(0,1)^r$, and $n\geq r$,
define the deterministic sliding-frequency map
\begin{equation}\label{eq:empirical-frequency-map}
  \mathsf P_{n,r,C}:(0,1)^n\to[0,1],
  \qquad
  \mathsf P_{n,r,C}(y_1^n)
  :=
  \frac{1}{n-r+1}
  \sum_{i=0}^{n-r}
  \1\{y_{i+1}^{i+r}\in C\}.
\end{equation}
It is Borel measurable because it is a finite average of indicators of
Borel cylinder sets.  Along the observed process we use the shorthand
\begin{equation}\label{eq:empirical-frequency}
  \widehat P_{n,r}(C)
  :=
  \mathsf P_{n,r,C}(Y_1^n).
\end{equation}

We use Birkhoff's theorem in the following standard form.  If
$(S,\mathcal S,\mu,T)$ is a probability-preserving dynamical system and
$f\in L^1(\mu)$, then
\[
  \frac1N\sum_{i=0}^{N-1} f\circ T^i
  \longrightarrow
  \mathbb E_\mu[f\mid\mathcal I_T]
  \qquad \mu\text{-almost surely},
\]
where $\mathcal I_T$ is the $T$-invariant $\sigma$-field.  If $T$ is
ergodic, the limit equals the constant $\int f\,d\mu$ almost surely
\cite{Birkhoff1931}.

\begin{lemma}[Simultaneous ergodic convergence on dyadic cylinders]
\label{lem:birkhoff-countable}
There exists an event $\Omega_0$ with $\Pp(\Omega_0)=1$ such that, for
every $r,\ell\geq1$ and every $C\in\cD_{r,\ell}$,
\[
  \widehat P_{n,r}(C)
  \longrightarrow
  \Pp(Y_1^r\in C)
  \qquad\text{on }\Omega_0.
\]
\end{lemma}

\begin{proof}
Let $\mu:=\Pp\circ Y^{-1}$ be the law of the trajectory
$Y=(Y_t)_{t\geq1}$ on the canonical path space $(0,1)^{\N}$, and let
$T$ denote the left shift.  Stationarity makes $\mu$ $T$-invariant, and
ergodicity of the process means that $T$ is ergodic under $\mu$.  For
fixed $r,\ell,C$, define
\[
  f_C(y_1,y_2,\ldots):=\1\{y_1^r\in C\}.
\]
By Birkhoff's pointwise ergodic theorem~\cite{Birkhoff1931},
\[
  \frac{1}{n-r+1}\sum_{i=0}^{n-r} f_C(T^i y)
  \longrightarrow
  \int f_C\,d\mu
  =
  \Pp(Y_1^r\in C)
\]
for $\mu$-almost every $y$.  Pulling this probability-one statement back
under the trajectory map $\omega\mapsto(Y_1(\omega),Y_2(\omega),\ldots)$
gives the asserted almost-sure convergence of
\eqref{eq:empirical-frequency} on the original probability space.

The set of triples
\[
  \{(r,\ell,C):r,\ell\geq1,\ C\in\cD_{r,\ell}\}
\]
is countable, because it is a countable union of finite sets.  Intersecting
the corresponding probability-one events gives a single event
$\Omega_0$ on which all these convergences hold simultaneously.
\end{proof}

The next elementary comparison transfers summability of $\phi$ to
summability of $\alpha$.

\begin{lemma}\label{lem:alpha-le-phi}
For every $m\geq1$,
\[
  \alpha(m)\leq\phi(m).
\]
Consequently,
\[
  \sum_{m=1}^\infty\alpha(m)\leq\Phi.
\]
\end{lemma}

\begin{proof}
For events $A,B$ with $\Pp(A)>0$,
\[
  \abs{\Pp(A\cap B)-\Pp(A)\Pp(B)}
  =
  \Pp(A)\abs{\Pp(B\mid A)-\Pp(B)}
  \leq
  \abs{\Pp(B\mid A)-\Pp(B)}.
\]
If $\Pp(A)=0$, the left-hand side is zero.  Taking the relevant suprema
proves $\alpha(m)\leq\phi(m)$, and summing gives the second assertion.
\end{proof}

The comparison $\alpha\leq\phi$ is used only to obtain summable
covariance control in the finite-target upper bound; it is not used to
approximate or replace the target coefficient $\phi(m)$.

We now obtain the only quantitative probability bound needed in the
proof.

\begin{lemma}[Variance of a sliding cylinder frequency]
\label{lem:variance}
Assume $\Phi\leq L$.  Let $1\leq r\leq3L$ and let
$C\subseteq(0,1)^r$ be Borel.  If $n\geq2r$, then
\begin{equation}\label{eq:variance-bound}
  \operatorname{Var}\!\left(\widehat P_{n,r}(C)\right)
  \leq
  \frac{7L}{n}.
\end{equation}
Consequently, for every $\delta>0$,
\begin{equation}\label{eq:cheb}
  \Pp\!\left(
    \abs{\widehat P_{n,r}(C)-\Pp(Y_1^r\in C)}>\delta
  \right)
  \leq
  \frac{7L}{n\delta^2}.
\end{equation}
\end{lemma}

\begin{proof}
Put
\[
  Z_i:=\1\{Y_{i+1}^{i+r}\in C\},
  \qquad i\geq0,
\]
and $N:=n-r+1$.  The sequence $(Z_i)$ is stationary and
$0\leq Z_i\leq1$, so $\operatorname{Var}(Z_i)\leq1/4$.

For $1\leq h<r$, the two length-$r$ windows overlap.  By
Cauchy--Schwarz,
\[
  \abs{\operatorname{Cov}(Z_0,Z_h)}
  \leq
  \sqrt{\operatorname{Var}(Z_0)\operatorname{Var}(Z_h)}
  \leq\frac14.
\]
For $h\geq r$, put
\[
  U:=\{Y_1^r\in C\},
  \qquad
  V_h:=\{Y_{h+1}^{h+r}\in C\}.
\]
Then $U\in\cF_1^r$ and
$V_h\in\cF_{h+1}^{h+r}\subseteq\cF_{h+1}^{\infty}$.  Hence, taking
$j=r$ in \eqref{eq:alpha-m},
\begin{align*}
  \abs{\operatorname{Cov}(Z_0,Z_h)}
  &=
  \abs{\Pp(U\cap V_h)-\Pp(U)\Pp(V_h)}\\
  &\leq
  \alpha(h-r+1).
\end{align*}
The indexing is exact: the second window begins at time $h+1=r+(h-r+1)$,
so its separation from the first window matches the lag convention in
\eqref{eq:alpha-m}.

Therefore
\begin{align*}
  \operatorname{Var}\!\left(\sum_{i=0}^{N-1}Z_i\right)
  &=
  N\operatorname{Var}(Z_0)
  +
  2\sum_{h=1}^{N-1}(N-h)\operatorname{Cov}(Z_0,Z_h)\\
  &\leq
  \frac{N}{4}
  +
  2\sum_{h=1}^{N-1}(N-h)
    \abs{\operatorname{Cov}(Z_0,Z_h)}\\
  &\leq
  \frac{N}{4}
  +
  2N\sum_{h=1}^{r-1}\frac14
  +
  2N\sum_{h=r}^{\infty}\alpha(h-r+1)\\
  &\leq
  N\left(\frac r2+2\sum_{s=1}^{\infty}\alpha(s)\right).
\end{align*}
By Lemma~\ref{lem:alpha-le-phi},
\[
  \sum_{s\geq1}\alpha(s)\leq\Phi\leq L.
\]
Therefore
\[
  \operatorname{Var}(\widehat P_{n,r}(C))
  \leq
  \frac{r/2+2L}{N}.
\]
Because $r\leq3L$ and $n\geq2r$,
\[
  N=n-r+1\geq\frac n2,
\]
and thus
\[
  \operatorname{Var}(\widehat P_{n,r}(C))
  \leq
  \frac{2(3L/2+2L)}{n}
  =
  \frac{7L}{n}.
\]
By stationarity,
$\E[\widehat P_{n,r}(C)]=\Pp(Y_1^r\in C)$.  Chebyshev's inequality
therefore yields \eqref{eq:cheb}.
\end{proof}

\begin{remark}
The overlapping windows contribute the finite term of order $r$ in the
variance calculation.  No independence of sliding blocks is assumed.
When the windows are disjoint, the remaining covariance tail is controlled
by $\sum_m\alpha(m)$.  The hypothesis $\Phi\leq L$ is used only in this analysis; the estimator
and the deterministic schedule do not depend on $\Phi$ or on any upper
bound for it.
\end{remark}

\section{Construction of the estimator}
\label{sec:estimator}

For each complexity level $c\geq1$, set
\begin{equation}\label{eq:parameters}
  q_c:=2^{-c},
  \qquad
  \varepsilon_c:=c^{-3},
  \qquad
  \delta_c:=\frac{q_c\varepsilon_c}{8}.
\end{equation}
Let
\[
  \cG_c
  :=
  \left\{
    (r,C):
    1\leq r\leq3c,\ C\in\cD_{r,c}
  \right\}
\]
and write
\begin{equation}\label{eq:Kc}
  K_c:=\abs{\cG_c}
  =
  \sum_{r=1}^{3c}2^{\,2^{rc}}.
\end{equation}
This is finite and deterministic.  By
Lemma~\ref{lem:common-span}, every candidate score at level $c$ depends on
at most three frequencies indexed by this single class $\cG_c$.  Thus
simultaneous control of the $K_c$ coordinates controls every candidate score,
even though the number of tuples $(m,j,k,A,B)$ is much larger; in particular,
the later union bound pays only the factor $K_c$.

Choose a strictly increasing deterministic sequence of integers
$(N_c)_{c\geq1}$ such that
\begin{equation}\label{eq:Nc-abstract-condition}
  N_c\geq6c,
  \qquad
  \frac{7cK_c}{N_c\delta_c^2}\leq2^{-c}.
\end{equation}
This is the only growth requirement used in the proof.  Fix the schedule
recursively by
\begin{equation}\label{eq:Nc-recursion}
  N_c
  :=
  \max\left\{
    N_{c-1}+1,\,
    6c,\,
    \left\lceil\frac{7cK_c2^c}{\delta_c^2}\right\rceil
  \right\},
  \qquad N_0:=0.
\end{equation}
The constants are chosen for convenient slack and summability, not optimized.
The resulting explicit law-independent schedule grows extremely quickly and
is used only for the strong-consistency argument.  In particular,
$N_c\to\infty$.

\begin{remark}[Fully expanded schedule]\label{rem:expanded-schedule}
The recursion is completely explicit.  Since
$\delta_c^2=2^{-2c}c^{-6}/64$, \eqref{eq:Nc-recursion} can equivalently be
written as
\[
  N_c
  =
  \max\left\{
    N_{c-1}+1,\,
    6c,\,
    \left\lceil448\,K_c\,c^7\,2^{3c}\right\rceil
  \right\}.
\]
This expansion is not used later; the proof uses only the conceptual
condition \eqref{eq:Nc-abstract-condition}.
\end{remark}

Fix $c\geq1$ and $1\leq m\leq c$.  For
$1\leq j,k\leq c$, $A\in\cD_{j,c}$ and $B\in\cD_{k,c}$, put
$r=j+m+k-1$ and use the lifted events $A^\uparrow,B^\uparrow,J$ defined
in Section~\ref{sec:approximation}.  Define the deterministic score map
\[
  \mathsf D_{c,m}^{A,B}:(0,1)^{N_c}\to[0,1]
\]
on the whole level-$c$ sample space by
\begin{equation}\label{eq:empirical-score}
  \mathsf D_{c,m}^{A,B}(y_1^{N_c})
  :=
  \begin{cases}
  \displaystyle
  \abs{
    \frac{\mathsf P_{N_c,r,J_{m,j,k}(A,B)}(y_1^{N_c})}
         {\mathsf P_{N_c,r,A^\uparrow}(y_1^{N_c})}
    -
    \mathsf P_{N_c,r,B^\uparrow}(y_1^{N_c})
  },
  &
  \mathsf P_{N_c,r,A^\uparrow}(y_1^{N_c})\geq q_c,\\[2ex]
  0,
  &
  \mathsf P_{N_c,r,A^\uparrow}(y_1^{N_c})<q_c.
  \end{cases}
\end{equation}
Thus the cutoff branch at the data vector $y_1^{N_c}$ is active exactly
when
\begin{equation}\label{eq:empirical-cutoff}
  \mathsf P_{N_c,r,A^\uparrow}(y_1^{N_c})\geq q_c.
\end{equation}
All three frequencies in \eqref{eq:empirical-score} are computed from the
same collection of length-$r$ sliding windows.  Since
$J_{m,j,k}(A,B)\subseteq A^\uparrow$, the empirical numerator is at most
the empirical denominator on the cutoff branch; hence
$\mathsf D_{c,m}^{A,B}$ takes values in $[0,1]$.

Define the deterministic level-$m$ maximum
\begin{equation}\label{eq:phihat-level}
  \mathsf H_{c,m}(y_1^{N_c})
  :=
  \max_{\substack{
      1\leq j,k\leq c\\
      A\in\cD_{j,c},\,B\in\cD_{k,c}
  }}
  \mathsf D_{c,m}^{A,B}(y_1^{N_c}),
\end{equation}
and the level-$c$ sum map
\begin{equation}\label{eq:Theta}
  \mathsf T_c:(0,1)^{N_c}\to[0,c],
  \qquad
  \mathsf T_c(y_1^{N_c})
  :=
  \sum_{m=1}^c\mathsf H_{c,m}(y_1^{N_c}).
\end{equation}
For the random observations, we henceforth use the abbreviations
\[
  \widehat d_{c,m}(A,B)
  :=\mathsf D_{c,m}^{A,B}(Y_1^{N_c}),
  \qquad
  \widehat\phi_c(m)
  :=\mathsf H_{c,m}(Y_1^{N_c}),
  \qquad
  \Theta_c
  :=\mathsf T_c(Y_1^{N_c}).
\]
Whenever a limit in $c$ is taken with $m$ fixed, $c$ is understood to
range over integers $c\geq m$.  The maximum in
\eqref{eq:phihat-level} is equivalently the maximum over tuples satisfying
the cutoff \eqref{eq:empirical-cutoff}; tuples failing the cutoff contribute
zero, while taking $A=(0,1)^j$ always passes the cutoff and has denominator
one.

Finally, for arbitrary sample size $n$, let
\begin{equation}\label{eq:cn}
  c(n):=\max\{c\geq1:N_c\leq n\},
\end{equation}
with $c(n)=0$ if the set is empty.  If $c(n)=c\geq1$, write
\[
  \pi_{n,N_c}:(0,1)^n\to(0,1)^{N_c},
  \qquad
  \pi_{n,N_c}(y_1^n):=y_1^{N_c},
\]
for the coordinate projection onto the first $N_c$ observations.  Define
the actual $n$-sample statistic as the deterministic map
\begin{equation}\label{eq:bounded-estimator}
  \widehat\Phi_n^{(0,1)}:(0,1)^n\to[0,\infty),
  \qquad
  \widehat\Phi_n^{(0,1)}(y_1^n)
  :=
  \begin{cases}
    0, & c(n)=0,\\
    \mathsf T_{c(n)}
      \bigl(\pi_{n,N_{c(n)}}(y_1^n)\bigr), & c(n)\geq1.
  \end{cases}
\end{equation}
Thus observations with indices $N_{c(n)}+1,\ldots,n$ are deliberately
ignored: between two consecutive schedule points the estimator is the
composition of the preceding level map with a fixed coordinate projection.
This makes the statistic defined for every $n$, not only along the subsequence
$(N_c)$.

\paragraph{Notation guide.}
The proof uses five related but distinct objects.  The population coefficient
$\phi(m)$ is the target at lag $m$; $\phi_c^{(q)}(m)$ is a population
truncation used only in the proof; $\widehat\phi_c(m)$ is the empirical
level-$c$ coefficient estimate; $\Theta_c=\sum_{m=1}^c\widehat\phi_c(m)$
is the scheduled level sum; and $\widehat\Phi_n^{(0,1)}$ is the bounded-state
statistic defined for every sample size $n$.  The real-valued statistic
$\widehat\Phi_n$ is obtained by the Borel transformation in the proof of
Theorem~\ref{thm:main}.

\begin{remark}[Computational scope]\label{rem:computational}
The bound
\[
  K_c\leq3c\,2^{2^{3c^2}}
\]
shows that the level classes, and therefore the explicit schedule $N_c$, are
enormous.  The construction is intended to establish universal strong
consistency with a completely deterministic schedule; no claim of
computational efficiency is made.
\end{remark}

\begin{proposition}[Measurability and universality]\label{prop:measurable}
For every $n$,
\[
  \widehat\Phi_n^{(0,1)}:(0,1)^n\to[0,\infty)
\]
is Borel measurable.  Its definition is independent of the law of $Y$.
\end{proposition}

\begin{proof}
The maps $\mathsf P_{n,r,C}$ are Borel by
\eqref{eq:empirical-frequency-map}.  At every fixed $c$, the index set in
\eqref{eq:phihat-level} is finite and deterministic.  Hence the cutoff set
\[
  \left\{
    y_1^{N_c}:
    \mathsf P_{N_c,r,A^\uparrow}(y_1^{N_c})\geq q_c
  \right\}
\]
is Borel.  On that set the denominator in
\eqref{eq:empirical-score} is bounded below by the positive constant
$q_c$, so the quotient is Borel there; on the complementary Borel set
$\mathsf D_{c,m}^{A,B}$ is the constant zero.  Thus every score map
$\mathsf D_{c,m}^{A,B}$ is Borel on all of $(0,1)^{N_c}$.  Finite maxima
and finite sums show that $\mathsf H_{c,m}$ and
$\mathsf T_c:(0,1)^{N_c}\to[0,c]$ are Borel.

Because $(N_c)$ is strictly increasing and unbounded, the set in
\eqref{eq:cn} is finite for every $n$ and $c(n)$ is deterministic and
well-defined.  For fixed $n$ with $c(n)=c\geq1$, the projection
$\pi_{n,N_c}$ is continuous, hence Borel, and
\eqref{eq:bounded-estimator} is the Borel composition
$\mathsf T_c\circ\pi_{n,N_c}$.  The case $c(n)=0$ is the constant-zero
map.  Finally, all objects
$q_c,\varepsilon_c,\delta_c,K_c,N_c$, all dyadic search classes, and all
coordinate projections are deterministic; no feature of the unknown process
law occurs in their definition.  Universality means that this same sequence
of maps is used for every law; the probability-one set on which convergence
holds may, as usual, depend on the law.
\end{proof}

\section{Upper control when the extended sum is finite}
\label{sec:upper}

Only the upper bound uses quantitative dependence control.  Suppose in this
section that $\Phi<\infty$; then Lemma~\ref{lem:alpha-le-phi} supplies
summable covariance control for the growing finite classes.  The comparison
$c\geq\Phi$ used below is made only in the proof, after the law has been
fixed; it is never used to define $c(n)$, $N_c$, or any statistic.  Define
\begin{equation}\label{eq:Ec}
  E_c
  :=
  \left\{
    \max_{(r,C)\in\cG_c}
    \abs{\widehat P_{N_c,r}(C)-\Pp(Y_1^r\in C)}
    \leq\delta_c
  \right\}.
\end{equation}

\begin{proposition}[Eventual simultaneous accuracy]\label{prop:uniform}
If $\Phi<\infty$, then
\[
  E_c
  \quad\text{occurs for all sufficiently large $c$, almost surely.}
\]
\end{proposition}

\begin{proof}
Set
\[
  c_0:=\max\{1,\lceil\Phi\rceil\}.
\]
For every integer $c\geq c_0$, every $(r,C)\in\cG_c$ satisfies
$r\leq3c$ and $N_c\geq6c\geq2r$.  Therefore
Lemma~\ref{lem:variance}, applied with $L=c$, and the union bound give
\[
  \Pp(E_c^\complement)
  \leq
  K_c\frac{7c}{N_c\delta_c^2}.
\]
By the deterministic choice \eqref{eq:Nc-abstract-condition},
\[
  \Pp(E_c^\complement)
  \leq
  \frac{7cK_c}{N_c\delta_c^2}
  \leq
  2^{-c}.
\]
Thus
\[
  \sum_{c=c_0}^{\infty}\Pp(E_c^\complement)<\infty.
\]
By the first Borel--Cantelli lemma (if $\sum_c\Pp(F_c)<\infty$, then
$\Pp(F_c\ \mathrm{i.o.})=0$), only finitely many $E_c^\complement$ occur
almost surely.  No independence among the events $E_c$ is required.
\end{proof}

The next deterministic inequality is the key to the rare-event issue.

\begin{lemma}[Stable division above an empirical cutoff]
\label{lem:stable-division}
Let $p,u,\widehat p,\widehat u\in[0,1]$ satisfy
\[
  0\leq u\leq p,
  \qquad
  \widehat p\geq q>0,
  \qquad
  \abs{\widehat p-p}\leq\delta,
  \qquad
  \abs{\widehat u-u}\leq\delta,
\]
where $\delta<q$.  Then $p>0$ and
\begin{equation}\label{eq:ratio-stability}
  \abs{\frac{\widehat u}{\widehat p}-\frac up}
  \leq
  \frac{2\delta}{q}.
\end{equation}
\end{lemma}

\begin{proof}
Since
\[
  p\geq\widehat p-\delta\geq q-\delta>0,
\]
the population ratio is defined.  Moreover,
\begin{align*}
  \abs{\frac{\widehat u}{\widehat p}-\frac up}
  &\leq
  \frac{\abs{\widehat u-u}}{\widehat p}
  +
  u\abs{\frac1{\widehat p}-\frac1p}\\
  &=
  \frac{\abs{\widehat u-u}}{\widehat p}
  +
  \frac up\frac{\abs{p-\widehat p}}{\widehat p}\\
  &\leq
  \frac{\delta}{q}+\frac{\delta}{q},
\end{align*}
because $u/p\leq1$.
\end{proof}

\begin{remark}[Empirical versus population cutoffs]
The cutoff $\widehat p\geq q$ is empirical, not an \emph{a priori}
assumption on the unknown population probability.  On the good event it
implies $p\geq q-\delta$; in particular, at level $c$ every admitted tuple
satisfies $p\geq(7/8)q_c$.  More importantly, $u\leq p$ makes the ratio
perturbation depend on the deterministic scale $q$, not on the unknown
denominator $p$.  Fixed lower-bound witnesses have $p>0$ and therefore
eventually pass the cutoff because $q_c\to0$.
\end{remark}

\begin{proposition}[Uniform levelwise upper control]\label{prop:upper}
If $\Phi<\infty$, then, almost surely, for all sufficiently large $c$ and
all $1\leq m\leq c$,
\begin{equation}\label{eq:coordinate-upper}
  \widehat\phi_c(m)\leq\phi(m)+\varepsilon_c.
\end{equation}
Consequently,
\[
  \limsup_{c\to\infty}\Theta_c\leq\Phi
  \qquad\text{almost surely},
\]
and for every fixed $m$,
\[
  \limsup_{c\to\infty}\widehat\phi_c(m)\leq\phi(m)
  \qquad\text{almost surely}.
\]
\end{proposition}

\begin{proof}
Work on the probability-one event on which $E_c$ holds eventually, as
provided by Proposition~\ref{prop:uniform}.  Fix a sufficiently large $c$, an
$m\leq c$, and any tuple $(j,k,A,B)$ in the finite index class of
\eqref{eq:phihat-level}.  Put
\[
  r=j+m+k-1.
\]
Since $j,m,k\leq c$, we have $r\leq3c-1<3c$.  Moreover, by the lifting
construction in Section~\ref{sec:approximation},
$A^\uparrow,B^\uparrow,J_{m,j,k}(A,B)\in\cD_{r,c}$.  Consequently all
three empirical probabilities used below are coordinates controlled
simultaneously by $E_c$.
If $\widehat P_{N_c,r}(A^\uparrow)<q_c$, then
$\widehat d_{c,m}(A,B)=0\leq\phi(m)+\varepsilon_c$ by definition.
It therefore remains to consider a tuple satisfying the cutoff
\eqref{eq:empirical-cutoff}.  Abbreviate
\begin{align*}
  p&:=\Pp(A^\uparrow),&
  \widehat p&:=\widehat P_{N_c,r}(A^\uparrow),\\
  s&:=\Pp(B^\uparrow),&
  \widehat s&:=\widehat P_{N_c,r}(B^\uparrow),\\
  u&:=\Pp(J_{m,j,k}(A,B)),&
  \widehat u&:=\widehat P_{N_c,r}(J_{m,j,k}(A,B)).
\end{align*}
Since $J_{m,j,k}(A,B)\subseteq A^\uparrow$, we have $0\leq u\leq p$.
The cutoff gives $\widehat p\geq q_c$.  On $E_c$,
\[
  \abs{\widehat p-p},
  \abs{\widehat s-s},
  \abs{\widehat u-u}
  \leq\delta_c.
\]
Because
\[
  \delta_c/q_c=\varepsilon_c/8<1,
\]
Lemma~\ref{lem:stable-division} yields
\[
  \abs{\frac{\widehat u}{\widehat p}-\frac up}
  \leq
  \frac{2\delta_c}{q_c}
  =
  \frac{\varepsilon_c}{4}.
\]
Therefore
\begin{align*}
  \widehat d_{c,m}(A,B)
  &\leq
  \abs{\frac up-s}
  +
  \abs{\frac{\widehat u}{\widehat p}-\frac up}
  +
  \abs{\widehat s-s}\\
  &\leq
  d_m(A,B)
  +
  \frac{\varepsilon_c}{4}
  +
  \delta_c.
\end{align*}
Since $q_c\leq1$,
\[
  \delta_c=\frac{q_c\varepsilon_c}{8}
  \leq\frac{\varepsilon_c}{8},
\]
and the events defining $d_m(A,B)$ belong to
$\cF_1^j$ and $\cF_{j+m}^{\infty}$, respectively.  Hence
$d_m(A,B)\leq\phi(m)$ and
\[
  \widehat d_{c,m}(A,B)
  \leq
  d_m(A,B)+\frac{3\varepsilon_c}{8}
  \leq
  \phi(m)+\frac{3\varepsilon_c}{8}
  \leq
  \phi(m)+\varepsilon_c.
\]
Taking the maximum over the entire finite index class gives
\eqref{eq:coordinate-upper}, simultaneously for every $m\leq c$ once
$E_c$ holds.  Summing over $m\leq c$,
\[
  \Theta_c
  \leq
  \sum_{m=1}^c\phi(m)+c\varepsilon_c
  \leq
  \Phi+c^{-2}.
\]
Taking $\limsup$ gives the asserted extended-sum upper bound.  Finally, for any
fixed $m$, \eqref{eq:coordinate-upper} holds for all sufficiently large
$c$ and $\varepsilon_c\to0$, which gives the coordinatewise upper bound.
\end{proof}

\section{The lower bound under mere ergodicity}
\label{sec:lower}

Unlike the upper bound, lower recovery requires no growing-class
uniformity or quantitative ergodic rate.  Fixed positive-probability
witnesses and Birkhoff convergence along $N_c\to\infty$ suffice, so the
argument also applies when $\Phi=\infty$.

\begin{proposition}[Pointwise lower recovery]\label{prop:pointwise-lower}
For every fixed $m\geq1$,
\[
  \liminf_{c\to\infty}\widehat\phi_c(m)
  \geq
  \phi(m)
  \qquad\text{almost surely}.
\]
No finiteness assumption on $\Phi$ is required.
\end{proposition}

\begin{proof}
Work on the probability-one event $\Omega_0$ of
Lemma~\ref{lem:birkhoff-countable}.  Fix $\eta>0$.  By
Proposition~\ref{prop:population-truncation}, applied to the estimator's sequence
$q_c=2^{-c}$, there exists a level $c_0\geq m$ such that
\[
  \phi_{c_0}^{(q_{c_0})}(m)>\phi(m)-\eta.
\]
Choose a maximizing tuple in the finite class defining
$\phi_{c_0}^{(q_{c_0})}(m)$.  Thus there are
$1\leq j,k\leq c_0$ and
$A\in\cD_{j,c_0}$, $B\in\cD_{k,c_0}$ such that, with
\[
  p:=\Pp(A^\uparrow)\geq q_{c_0}>0,
\]
we have
\begin{equation}\label{eq:witness}
  d_m(A,B)>\phi(m)-\eta.
\end{equation}
The integers $m,c_0,j,k$ and the pair $(A,B)$ are now fixed.  This
population witness may depend on the unknown law of the process; it is used
only in the consistency proof and does not enter the definition of the
estimator.

For every sufficiently large $c$,
\[
  c\geq c_0,
\]
so dyadic refinement puts the same sets $A$ and $B$ in the level-$c$ search
class.  Their population probabilities and the population score
$d_m(A,B)$ are independent of $c$; only their representation inside the
finer finite search class changes.  The block span
\[
  r=j+m+k-1
\]
is fixed.  Lemma~\ref{lem:birkhoff-countable} gives convergence of each
relevant frequency as the sample size tends to infinity.  Since $N_c\to
\infty$, the same convergence holds along the deterministic subsequence
$n=N_c$, and therefore
\begin{align*}
  \widehat P_{N_c,r}(A^\uparrow)&\to p,\\
  \widehat P_{N_c,r}(B^\uparrow)&\to \Pp(B^\uparrow),\\
  \widehat P_{N_c,r}(J_{m,j,k}(A,B))
  &\to\Pp(J_{m,j,k}(A,B)).
\end{align*}
Because $p>0$, the first convergence gives
$\widehat P_{N_c,r}(A^\uparrow)\geq p/2$ eventually; because
$q_c\to0$, also $q_c<p/2$ eventually.  Therefore
\[
  \widehat P_{N_c,r}(A^\uparrow)\geq q_c
\]
for all sufficiently large $c$.
Thus the cutoff branch of \eqref{eq:empirical-score} is eventually
active for the fixed witness $(A,B)$.  Its empirical score therefore
converges to $d_m(A,B)$, and this witness is among the tuples maximized
over in \eqref{eq:phihat-level}.  Consequently,
\[
  \liminf_{c\to\infty}\widehat\phi_c(m)
  \geq d_m(A,B)
  >
  \phi(m)-\eta.
\]
All required cylinder-frequency convergences hold on the same event
$\Omega_0$, independently of the chosen fixed non-random witness.  Applying
the preceding argument with $\eta=1/s$, $s\in\N$, and then letting
$s\to\infty$ proves the claim on $\Omega_0$.  Since $m$ was arbitrary and
$\Omega_0$ already contains the Birkhoff convergence for the entire
countable dyadic cylinder family, the conclusion holds simultaneously for
every fixed $m\geq1$ on this same event.
\end{proof}

\begin{corollary}[Consistency of each fixed coefficient under summability]
\label{cor:coordinate-consistency}
If $\Phi<\infty$, then for every fixed $m\geq1$,
\[
  \widehat\phi_c(m)\longrightarrow\phi(m)
  \qquad\text{almost surely}.
\]
\end{corollary}

\begin{proof}
Intersect the probability-one events in Propositions~\ref{prop:upper} and~\ref{prop:pointwise-lower}.
On that event,
\[
  \phi(m)
  \leq
  \liminf_{c\to\infty}\widehat\phi_c(m)
  \leq
  \limsup_{c\to\infty}\widehat\phi_c(m)
  \leq
  \phi(m).
\]
\end{proof}

\begin{proposition}[Lower bound for the extended sum]\label{prop:sum-lower}
For every stationary ergodic process,
\[
  \liminf_{c\to\infty}\Theta_c\geq\Phi
\]
in the extended-real sense.
\end{proposition}

\begin{proof}
Work on the probability-one event $\Omega_0$ from
Lemma~\ref{lem:birkhoff-countable}.  By Proposition~\ref{prop:pointwise-lower},
with the simultaneous-event observation at the end of its proof, the required
pointwise lower bound holds on $\Omega_0$ for every fixed $m\geq1$.
Fix $\omega\in\Omega_0$; all limits and inequalities below are evaluated at
this $\omega$.

Fix $M\geq1$.  Since every empirical coefficient is nonnegative, for
$c\geq M$,
\[
  \Theta_c
  =
  \sum_{m=1}^c\widehat\phi_c(m)
  \geq
  \sum_{m=1}^M\widehat\phi_c(m).
\]
For a finite sum,
\[
  \liminf_{c\to\infty}
  \sum_{m=1}^M\widehat\phi_c(m)
  \geq
  \sum_{m=1}^M
  \liminf_{c\to\infty}\widehat\phi_c(m).
\]
Using Proposition~\ref{prop:pointwise-lower},
\[
  \liminf_{c\to\infty}\Theta_c
  \geq
  \sum_{m=1}^M\phi(m).
\]
Finally let $M\to\infty$.  Since the partial sums are nondecreasing,
\[
  \sup_M\sum_{m=1}^M\phi(m)=\Phi,
\]
including the case $\Phi=\infty$.
\end{proof}

\section{Completion of the proof}
\label{sec:completion}

\begin{proposition}[Consistency along the complexity schedule]
\label{prop:schedule-consistency}
For every $(0,1)$-valued stationary ergodic process,
\[
  \Theta_c\longrightarrow\Phi
  \qquad\text{almost surely in }[0,\infty].
\]
\end{proposition}

\begin{proof}
If $\Phi<\infty$, work on the intersection of the probability-one events
supplied by Propositions~\ref{prop:upper} and~\ref{prop:sum-lower}.  On this common event,
\[
  \Phi
  \leq
  \liminf_{c\to\infty}\Theta_c
  \leq
  \limsup_{c\to\infty}\Theta_c
  \leq
  \Phi.
\]
Hence $\Theta_c\to\Phi$ almost surely.

If $\Phi=\infty$, Proposition~\ref{prop:sum-lower} gives
$\liminf_{c\to\infty}\Theta_c=+\infty$, which is exactly
$\Theta_c\to+\infty$ in the extended half-line.
\end{proof}

\begin{proof}[Proof of Theorem~\ref{thm:main}]
Let
\[
  h(x):=\frac12+\frac1\pi\arctan x,
  \qquad x\in\R,
\]
as in Proposition~\ref{prop:reduction}.  For every $n\geq1$ and
$(x_1,\ldots,x_n)\in\R^n$, define the real-valued-data statistic
explicitly by
\begin{equation}\label{eq:real-estimator}
  \widehat\Phi_n(x_1,\ldots,x_n)
  :=
  \widehat\Phi_n^{(0,1)}
  \bigl(h(x_1),\ldots,h(x_n)\bigr),
\end{equation}
where $\widehat\Phi_n^{(0,1)}$ is the deterministic bounded-state
statistic in \eqref{eq:bounded-estimator}.  By
Proposition~\ref{prop:measurable} and the Borel measurability of $h$, the map in
\eqref{eq:real-estimator} is Borel measurable on $\R^n$, and its
definition is independent of the law of $X$.

Now put $Y_t=h(X_t)$.  By Proposition~\ref{prop:reduction}, $Y$ is stationary and ergodic and
$\Phi_Y=\Phi_X$.  Because $(N_c)$ is strictly increasing and unbounded,
$c(n)\to\infty$.
For every $n$ with $c(n)\geq1$, the definition
\eqref{eq:bounded-estimator} gives the pathwise identity
\[
  \widehat\Phi_n^{(0,1)}(Y_1,\ldots,Y_n)=\Theta_{c(n)}.
\]
Therefore Proposition~\ref{prop:schedule-consistency} and
$c(n)\to\infty$ imply
\[
  \widehat\Phi_n^{(0,1)}(Y_1,\ldots,Y_n)
  =\Theta_{c(n)}
  \longrightarrow
  \Phi_Y
  \qquad\text{almost surely}.
\]
Using \eqref{eq:real-estimator} and $\Phi_Y=\Phi_X$ therefore yields
\[
  \widehat\Phi_n(X_1,\ldots,X_n)
  \longrightarrow
  \Phi_X
  \qquad\text{almost surely}.
\]
This proves all assertions of Theorem~\ref{thm:main}.
\end{proof}

\begin{corollary}[Universal consistency under any fixed lag shift]
\label{cor:shifted}
Fix \(d\in\{0,1,2,\ldots\}\), and define
\[
  \phi_X^{[d]}(m)
  :=
  \sup_{j\geq1}
  \phi\bigl(\cF_1^j,\cF_{j+m+d}^{\infty}\bigr),
  \qquad m\geq1,
\]
and
\[
  \Phi_X^{[d]}
  :=
  \sum_{m=1}^{\infty}\phi_X^{[d]}(m)
  \in[0,\infty].
\]
There exists a deterministic sequence of Borel measurable functions
\[
  \widehat\Phi_n^{[d]}:\mathbb R^n\to[0,\infty)
\]
such that, for every real-valued discrete-time stationary ergodic process
\(X\),
\[
  \widehat\Phi_n^{[d]}(X_1,\ldots,X_n)
  \longrightarrow
  \Phi_X^{[d]}
  \qquad\text{almost surely in }[0,\infty].
\]
Thus the conclusion holds in particular for the future beginning at
\(j+m+1\).
\end{corollary}

\begin{proof}
By definition,
\[
  \phi_X^{[d]}(m)=\phi_X(m+d),
  \qquad
  \Phi_X^{[d]}=\sum_{r=d+1}^{\infty}\phi_X(r).
\]
The case \(d=0\) is Theorem~\ref{thm:main}; the argument below also makes the
tail construction explicit.  For the bounded-state construction, put
\[
  \Theta_c^{[d]}
  :=
  \begin{cases}
    0, & c\leq d,\\[1mm]
    \displaystyle\sum_{r=d+1}^{c}\widehat\phi_c(r), & c>d.
  \end{cases}
\]
If \(\Phi^{[d]}<\infty\), then the omitted prefix is finite:
\[
  0\leq\sum_{r=1}^{d}\phi(r)\leq d.
\]
Consequently,
\[
  \Phi
  =
  \sum_{r=1}^{d}\phi(r)+\Phi^{[d]}
  <\infty,
\]
with the prefix interpreted as zero when \(d=0\).  Hence
Proposition~\ref{prop:upper} applies.
On its probability-one event, for all sufficiently large \(c>d\),
\[
  \Theta_c^{[d]}
  \leq
  \sum_{r=d+1}^{c}\phi(r)+(c-d)\varepsilon_c
  \leq
  \Phi^{[d]}+c\varepsilon_c,
\]
and \(c\varepsilon_c=c^{-2}\to0\).  Therefore
\[
  \limsup_{c\to\infty}\Theta_c^{[d]}
  \leq
  \Phi^{[d]}.
\]

For the reverse inequality, work on the single event \(\Omega_0\) from
Lemma~\ref{lem:birkhoff-countable}.  Fix \(M>d\).  For every \(c\geq M\),
nonnegativity gives
\[
  \Theta_c^{[d]}
  \geq
  \sum_{r=d+1}^{M}\widehat\phi_c(r).
\]
Since the sum on the right is finite,
Proposition~\ref{prop:pointwise-lower} yields
\[
  \liminf_{c\to\infty}\Theta_c^{[d]}
  \geq
  \sum_{r=d+1}^{M}
  \liminf_{c\to\infty}\widehat\phi_c(r)
  \geq
  \sum_{r=d+1}^{M}\phi(r).
\]
Letting \(M\to\infty\) yields
\[
  \liminf_{c\to\infty}\Theta_c^{[d]}
  \geq
  \Phi^{[d]},
\]
including the case \(\Phi^{[d]}=\infty\).  Thus
\[
  \Theta_c^{[d]}\longrightarrow\Phi^{[d]}
  \qquad\text{almost surely in }[0,\infty].
\]

It remains only to define the statistic for every sample size.  Let \(h\)
be the Borel bijection from Proposition~\ref{prop:reduction}, and let
\(c(n)\) be as in \eqref{eq:cn}.  If
\(c(n)\leq d\), set \(\widehat\Phi_n^{[d]}=0\).  If \(c(n)=c>d\), define
\[
  \widehat\Phi_n^{[d]}(x_1,\ldots,x_n)
  :=
  \sum_{r=d+1}^{c}
  \mathsf H_{c,r}
  \bigl(h(x_1),\ldots,h(x_{N_c})\bigr).
\]
This is a deterministic Borel function by the same measurability argument as
in Proposition~\ref{prop:measurable}.  Along the observed process it equals
\(\Theta_{c(n)}^{[d]}\); since \(c(n)\to\infty\), the asserted convergence
follows.
\end{proof}

\section{Discussion}\label{sec:discussion}

Classical universal conditional-prediction results under stationarity and
ergodicity do not subsume the theorem.  Ornstein~\cite{Ornstein1978},
Algoet~\cite{Algoet1992}, and Morvai, Yakowitz, and Gy\"orfi
\cite{MorvaiYakowitzGyorfi1996} concern prediction or conditional laws at the
realized past.  By contrast, $\phi(m)$ is a worst-case supremum over all
positive-probability past events and all future events at the prescribed
separation.

For comparison, McDonald, Shalizi, and
Schervish~\cite{McDonaldEtAl2011,McDonaldEtAl2015} study estimation of
$\beta$-mixing coefficients from a stationary sample path, and Khaleghi and
Lugosi~\cite{KhaleghiLugosi2023} estimate the $\ell_1$-norms of the
$\alpha$- and $\beta$-mixing sequences universally.  One-trajectory results
under Markov or other structural assumptions include
\cite{HsuEtAl2019,WolferKontorovich2019,Wolfer2020,WolferAlquier2024};
these do not yield the unrestricted $\phi$-sum theorem above.

The exact sum has operational motivation.  In the $\phi$-mixing
restless-bandit model of Gr\"unew\"alder and
Khaleghi~\cite{GrunewalderKhaleghi2019}, the dependence correction is
governed by $\|\phi\|_1=\sum_m\phi(m)$ or an upper bound on it.  More
broadly, mixing coefficients enter concentration and learning bounds for
dependent data; see, for example, Rio~\cite{Rio2000,Rio2000book},
Samson~\cite{Samson2000}, Mohri and Rostamizadeh~\cite{MohriRostamizadeh2010},
and Alquier and Wintenberger~\cite{AlquierWintenberger2012}.

\appendix
\section{One-sided and two-sided conventions}\label{app:two-sided}

We prove the claim stated in Remark~\ref{rem:two-sided}.  Let
$\widetilde X=(\widetilde X_t)_{t\in\mathbb Z}$ be a stationary two-sided
extension in distribution of $X$.  To construct one, for each finite
$I\subset\mathbb Z$ shift $I$ far enough to the right into $\mathbb N$ and
use the corresponding finite-dimensional law of $X$.  Stationarity makes
this law independent of the chosen shift.  The resulting family is
permutation-consistent and projectively consistent; since $\mathbb R$ is a
standard Borel space and the index set $\mathbb Z$ is countable, the
Kolmogorov extension theorem yields a probability law on $\mathbb R^{\mathbb
Z}$ with these finite-dimensional marginals; see, for example,
Kallenberg~\cite{Kallenberg2021}.  The resulting law is stationary by
construction.  Because probability laws on the countable product
$\mathbb R^{\mathbb Z}$ are determined by their finite-dimensional
distributions (equivalently, cylinder sets generate the product
$\sigma$-field), these stationary finite-dimensional identities also
determine the joint law of each finite past block together with the entire
countable future sequence.  Hence stationarity may be applied below even
though the second $\sigma$-field is the infinite-future $\sigma$-field.
Then
\[
  \phi_X(m)=\sup_{j\geq1}
  \phi\bigl(\sigma(\widetilde X_{-j+1}^0),\sigma(\widetilde X_t:t\geq m)\bigr)
  =\phi\bigl(\sigma(\widetilde X_t:t\leq0),\sigma(\widetilde X_t:t\geq m)\bigr).
\]
The first equality follows by stationarity, shifting the terminal coordinate
of the finite past block from $j$ to $0$.  For the second, write
\[
  \cC_j:=\sigma(\widetilde X_{-j+1}^0),\qquad
  \cC_\infty:=\sigma(\widetilde X_t:t\leq0),\qquad
  \cB_m:=\sigma(\widetilde X_t:t\geq m).
\]
Since $\cC_j\subseteq\cC_\infty$,
\[
  \sup_{j\geq1}\phi(\cC_j,\cB_m)
  \leq \phi(\cC_\infty,\cB_m).
\]
For the reverse inequality, fix $A\in\cC_\infty$ with
$p:=\Pp(A)>0$.  The union $\bigcup_{j\geq1}\cC_j$ is an algebra generating
$\cC_\infty$, so the generating-algebra approximation of
Lemma~\ref{lem:measure-approx} yields events
$A_r\in\bigcup_{j\geq1}\cC_j$ with
$d_r:=\Pp(A_r\triangle A)\to0$.  Put $p_r:=\Pp(A_r)$.  Then
$p_r\to p$, so $p_r>0$ eventually.  Moreover, for every
$B\in\cB_m$ and all such $r$,
\begin{align*}
  \abs{\Pp(B\mid A_r)-\Pp(B\mid A)}
  &\leq
  \frac{\Pp(A_r\triangle A)}{p_r}
  +
  \frac{\abs{p_r-p}}{p_r}\\
  &\leq
  \frac{2d_r}{p_r}
  \longrightarrow0.
\end{align*}
The bound is uniform in the future event $B$.  Hence
\[
  \abs{\Pp(B\mid A_r)-\Pp(B)}
  \longrightarrow
  \abs{\Pp(B\mid A)-\Pp(B)}
\]
uniformly over $B\in\cB_m$.  Each $A_r$ belongs to some finite-past
$\sigma$-field $\cC_{j_r}$, so taking suprema first over $B$, then over all
positive-probability $A\in\cC_\infty$, gives
\[
  \phi(\cC_\infty,\cB_m)
  \leq
  \sup_{j\geq1}\phi(\cC_j,\cB_m).
\]
Thus equality holds, and the one-sided observation convention has the usual
stationary two-sided coefficient as its target.  In particular, this value is
intrinsic to the one-sided stationary law and does not depend on the chosen
two-sided extension.

\end{document}